\documentclass[10pt, one side,emlines]{amsart}
\usepackage{amssymb,latexsym,xy,eucal,mathrsfs}
\textwidth=15.2cm \textheight=24cm \theoremstyle{plain}
\usepackage{amssymb,latexsym,xy,eucal,mathrsfs,tikz}
\newtheorem{lem}{Lemma}[section]

\newtheorem{thm}[lem]{Theorem}

\theoremstyle{definition}

\begin{document}

\baselineskip 14truept
\title{On removal of perfect matching from folded hypercubes.}

\author{S. A. Mane }

\dedicatory{Center for Advanced Studies in Mathematics,
	Department of Mathematics,\\ Savitribai Phule Pune University, Pune-411007, India.\\
manesmruti@yahoo.com  \\}

  \maketitle

 \begin{abstract} The hypercube $Q_n$ of dimension $n$ is one of the most versatile and powerful interconnection networks. The $n-$dimensional folded cube denoted as $FQ_n$,
 	a variation of the hypercube possesses some embeddable properties that the hypercube does not possess. Dong and Wang\cite{dw}(In Theor. Comput. Sci.$ 771(2019)93-98)$ conjectured that "A subset $E^m$ of $2^n - 1$ edges of $FQ_n$ is a perfect matching if and only if $FQ_n - E^m$ is isomorphic to $Q_n$". In this paper, we disprove this conjecture by providing some perfect matchings removal of which from $FQ_n$ do not give a graph isomorphic to $Q_n$.

 \end{abstract}



\section{Introduction}  For undefined terminology and
notations, we refer to the reader to West\cite{we}. Given any two graphs $G$ and $H$,
an injection $f : V(G)  V(H)$ is an embedding of $G$ into $H$ if $f (u)$ and $f (v)$ are adjacent in $H$ whenever $u$ and $v$ are adjacent
in $G$. If $G$ embeds in $H$, then $G$ is isomorphic to a subgraph of $H$. If $f$ is a bijection, then $G$ is isomorphic to a spanning subgraph
of $H$, moreover, if $G$ and $H$ are same then $f$ is called as an automorphism. 
Implementation of a parallel algorithm on a parallel computer can be modeled as an embedding problem. Therefore,
spanning subgraphs of interconnection networks have been a subject of both theoretical and applied research.\\
 Hypercubes are widely studied as they meet
 several conflicting demands that arise in the design of
 interconnection networks. The machine based on hypercubes such
 as the Cosmic Cube from Caltech, the iPSC/2 from Intel and Connection Machines have been implemented
 commercially. Several variations of hypercubes
 have been proposed and investigated to improve the efficiency
 of hypercube networks. The $n-$dimensional folded cube denoted as $FQ_n$,
a variation of the hypercube $Q_n$. Folded hypercubes (folded cube) is a standard hypercube with some extra links established between the nodes \cite{c}. \\
A set of edges $M$ of a graph $G$ is matching if every vertex
of $G$ is incident with at most one edge of $M$. If a vertex $v$ of $G$ is incident
with an edge of $M$, we say that $v$ is covered by $M$. A matching $M$ is perfect ($1-$ factor) if
every vertex of $G$ is covered by $M$. Thus, perfect matching is a spanning subgraph (Papers \cite{gh} and \cite{o} focuses more on this topic especially in hypercubes).\\ 
Hence, studying embedding of $Q_n$ after the removal of perfect matching from $FQ_n$ is a very interesting topic to study. Dong and Wang\cite{dw} raised conjecture which states "A subset $E^m$ of $2^n - 1$ edges of $FQ_n$ is a perfect matching if and only if $FQ_n - E^m$ is isomorphic to $Q_n$". In this paper, we disprove this conjecture by constructing some perfect matchings in $FQ_n$ removing those from $FQ_n$ do not give a graph isomorphic to $Q_n$.

 \section {Preliminaries}
 
 An $n-$dimensional hypercube $Q_n$ can be represented as an undirected graph $Q_n = (V, E)$ such that $V$ consists of $2^n$
  nodes(vertices) which are labeled as binary numbers of length $n$. $E$ is the set of edges that connect two nodes if and only if 
  they  differ in exactly one bit of their labels. 
The \textit{parity} of a vertex in $Q_n$ is the parity of the number
of $1$s in its name, even or odd. If the number of even parity
vertices is the same as the number of odd parity vertices then the graph is said to be \textit{balanced}. The hypercube $Q_n$ is balanced.
 It has many attractive properties, such as being bipartite, $n-$regular, $n-$connected.
 The diameter of a graph $G$ (diam $G$) is the maximal eccentricity in G.
 The eccentricity of a vertex is its greatest distance to any other vertex. diam $(Q_n) = n.$
 It is a vertex-transitive(symmetric) graph in the sense that, for given any two vertices $v_1$ and $v_2$ of $Q_n$, there is some automorphism.
 
 One copy of $Q_{n+1}$ can be decomposed into two copies of $Q_{n}$ (denoted by $Q^0_n$ and $Q^1_n$) whose vertices are joined by $2^{n}$ edges of a perfect matching $R$. These edges are called parallel edges. Let $V(Q_n) =\{v_i: 1 \leq i \leq 2^n \}$ and
 without loss $V(Q_{n+1}) = \{(v_i, 0), (v_i, 1): v_i \in V(Q_n), 1
 \leq i \leq 2^n\}$. So, we write $Q_{n+1} =
Q^0_n\cup Q^1_n\cup R$ where $ V(Q^0_n)  =\{(v_i, 0): v_i \in V(Q_n), 1
 \leq i \leq 2^n\}$, $ V(Q^1_n)  =\{(v_i, 1): v_i \in
 V(Q_n), 1 \leq i \leq 2^n\}$. Note
 that the end vertices of
 any edge in $R$ are called corresponding vertices. 
  
  For  $k = n + m$. We decompose $Q_k = Q_n
  \Box Q_{m}$. Now for any $t \in V(Q_{m})$, we denote by
  $Q^t_{n}$ the subgraph of $Q_k$ induced by the vertices whose
  last ${m}$ components form the tuple $t$. In general, if $A \subseteq V(Q_n)$ and $B \subseteq V(Q_m)$ then we let $(A, B)$ be the subgraph of $Q_{n+m}$ induced by the vertex set $\{ (x, y) :  x \in A$ and $y\in B $. Further, if $B = \{t \}$ then we write $(A, \{t \}) = (A, t) $.  It is easy to observe that $Q^t_{n}$ is isomorphic to $Q_n$.
  
   The folded $n-$cube of dimension $n$ has $2^n$ vertices, each labeled by an
 $n-$bit binary string $(a_1,a_2, \ldots ,a_n)$. For $n \geq 2$, $FQ_n$ is obtained by taking two copies of the hypercube $Q^0_{(n-1)}$ and $Q^1_{(n-1)}$
  and adding $2 \times 2^{(n-1)}$ edges between the two as follows:\\
  A vertex $A = (a_1,a_2, \ldots ,a_{(n-1)}, 0)$ of $Q^0_{(n-1)}$ is joined to vertex $B = (b_1,b_2, \ldots ,b_{(n-1)}, 1)$ of $Q^1_{(n-1)}$ iff for every $i, 2  \leq i \leq  n,$ either \\     
   $(i)$ $a_i = b_i$; in this case, $AB$ is called a hypercube
  edge ($B$ is sometimes denoted as $A^h$), or\\
  $(ii)$  $a_i = \overline{b_i}$; in this case, $AB$ is called a complement
  edge($B$ is sometimes denoted as $A^c$).\\
 Thus, a folded $n-$cube, $FQ_n$ is obtained from $Q_n$ by adding additional link between two nodes whose addresses are complementary to each other.
   If there is a path between vertices say $u,v$ and $w$ then we denote it as $P= u-v-w $. Some perfect matchings in $FQ_n$ removal of those from $FQ_n$ do not give a graph isomorphic to $Q_n$, we call these matchings as non-removable matchings. Perfect matchings removal of which gives graph isomorphic to $Q_n$,  we call them removable matching.  
   
   \section {\bf Construction of removable and non-removable perfect matchings.}

 In this section, we construct removable and non-removable perfect matchings.  \\
  We decompose $Q^0_{(n-1)}$ and $Q^1_{(n-1)}$ each in ${(n-1)}^{th}$ direction which gives $Q^0_{(n-1)} = Q^{00}_{(n-2)} \cup Q^{10}_{(n-2)} \cup M_{00} $ and $Q^1_{(n-1)} = Q^{01}_{(n-2)} \cup Q^{11}_{(n-2)} \cup M_{11} $, where $M_{00}$ and $M_{11}$ both are set of sub hypercube edges(say) and $|M_{00} | = |M_{11}| = 2^{(n-2)}$. We denote by $M_0 = M_{00} \cup M_{11}$       
We write $FQ_n = Q^0_{(n-1)} \cup Q^1_{(n-1)} \cup M_1 \cup M_2 $, where $M_1$ and $M_2$ denote set of hypercube edges and complement(augmented) edges respectively. Clearly, $|M_1 | = |M_2| = 2^{(n-1)}$ and both are perfect matchings. Also, we can write $FQ_n = Q^{00}_{(n-2)} \cup Q^{10}_{(n-2)} \cup M_0 \cup Q^{01}_{(n-2)} \cup Q^{11}_{(n-2)} \cup M_1 \cup M_2 $ (see FIGURE.$1$), it can be easily observe that $M_2$ contain edge set which joins end vertices of $Q^{00}_{(n-2)}$ to $Q^{11}_{(n-2)}$ and end vertices of $Q^{01}_{(n-2)}$ to $Q^{10}_{(n-2)}$.\\
\begin{figure}[h]
	\begin{tikzpicture} [scale=.5]    
	
	\node [below] at (0,0) {$Q^{10}_{n-2}$}; 
	\node [below] at (0,10) {$Q^{00}_{n-2}$};  
	\node [below] at (10,10) {$Q^{01}_{n-2}$};  
	\node [below] at (10,0) {$Q^{11}_{n-2}$}; 
	
	\node [below] at (0,-2.6) {$Q^0_{n-1}$};  
	\node [below] at (10,-2.6) {$Q^1_{n-1}$};
	
	\node [above] at (5,10.5) {Part of $M_1$};
	\node [below] at (5,-1.5) {Part of $M_1$};
	
	\node [below, rotate=50] at (5.8,7) {Part of $M_2$};
	\node [below, rotate=-50] at (4,7.2) {Part of $M_2$};
	
	\node [left] at (-0.3,5) {$M_{00}$};  
	\node [right] at (10.1,5) {$M_{11}$};
	
	\draw (0.82,-1)--(9.03,-1);  
	\draw (0.3,-1.4)--(9.6,-1.4); 
	
	\draw (0.82,10.2)--(9.02,10.2);  
	\draw (0.3,10.6)--(9.6,10.6);
	
	\draw (-0.4,8.63)--(-0.4,0.57);  
	\draw (0.3,8.61)--(0.3,0.59);
	
	\draw (9.6,8.58)--(9.6,0.6);  
	\draw (10.2,8.58)--(10.2,0.6);
	
	\draw (0.78,0.25)--(8.85,9.4);
	\draw (1.01,-0.5)--(9.2,8.8);  
	
	\draw (9.17,0.33)--(1,9.5);
	\draw (8.86,-0.4)--(0.7,8.83);

	\draw[thick, rotate=90] (9.6,0.03) ellipse (30pt and 30pt);
	\draw[thick, rotate=90] (9.6,-9.9) ellipse (30pt and 30pt);
	\draw[thick, rotate=90] (-0.4,0.03) ellipse (30pt and 30pt);
	\draw[thick, rotate=90] (-0.4,-9.9) ellipse (30pt and 30pt);
	
	\draw[thick, rotate=90] (4.5,0) ellipse (200pt and 80pt);
	\draw[thick, rotate=90] (4.5,-9.8) ellipse (200pt and 80pt);
	\end{tikzpicture}
	\caption{Folded cube $FQ_n$}
\end{figure}
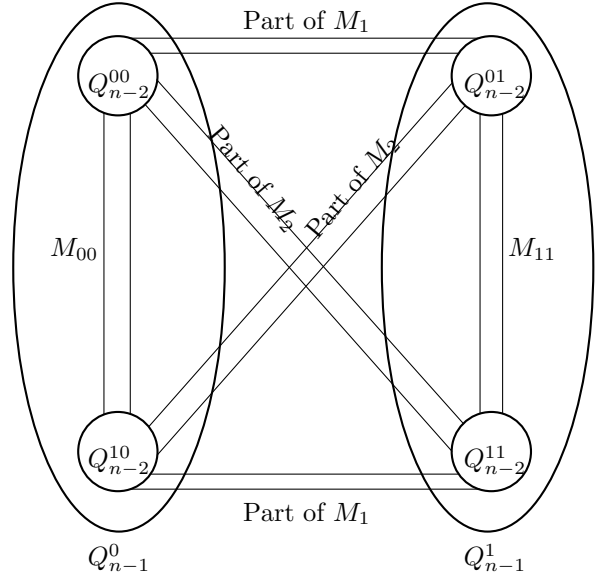

\begin{thm} Let $n \geq 2$. In $FQ_n$, a perfect matching $M$ is removable if $M = M_i (0 \leq i \leq 2)$.  
\end{thm} 
\begin{proof} We want to prove that there exists some $f$ from  $V(Q_n)$  to $V(FQ_n - M) (M = M_i (0 \leq i \leq 2))$ with the property that $f$ is bijection and $f(u)$ and $f(v)$ are adjacent in $FQ_n - M$ whenever $u$ and $v$ are adjacent $Q_n$. Since both the graphs are $n-$regular
	and have the same vertex set implies $f$ is bijective. It is enough to show that
	$f(X)f(Y) \in E(FQ_n - M)$ whenever $XY \in E(Q_n)$.\\
Case $1: $ $M = M_2$.\\	
	In this case $f : V(FQ_n - M_2) \rightarrow V(Q_n)$, we define as $f(x_1,x_2,....x_n) = (x_1,x_2,....x_n) $   
	 one can easily observe that it is an isomorphism(see FIGURE.$2$).\\
	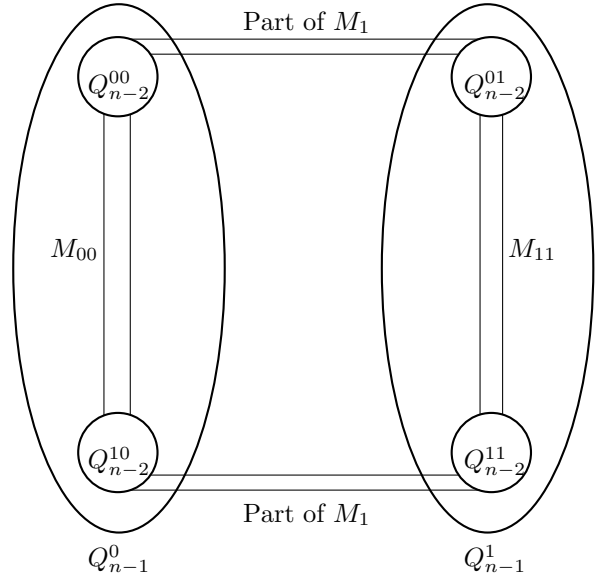
\begin{figure}[h]
		\begin{tikzpicture} [scale=.5]    
		
		\node [below] at (0,0) {$Q^{10}_{n-2}$}; 
		\node [below] at (0,10) {$Q^{00}_{n-2}$};  
		\node [below] at (10,10) {$Q^{01}_{n-2}$};  
		\node [below] at (10,0) {$Q^{11}_{n-2}$}; 
		
		\node [below] at (0,-2.6) {$Q^0_{n-1}$};  
		\node [below] at (10,-2.6) {$Q^1_{n-1}$};
		
		\node [above] at (5,10.5) {Part of $M_1$};
		\node [below] at (5,-1.5) {Part of $M_1$};
		
		
		\node [left] at (-0.3,5) {$M_{00}$};  
		\node [right] at (10.1,5) {$M_{11}$};
		
		\draw (0.82,-1)--(9.03,-1);  
		\draw (0.3,-1.4)--(9.6,-1.4); 
		
		\draw (0.82,10.2)--(9.02,10.2);  
		\draw (0.3,10.6)--(9.6,10.6);
		
		\draw (-0.4,8.63)--(-0.4,0.57);  
		\draw (0.3,8.61)--(0.3,0.59);
		
		\draw (9.6,8.58)--(9.6,0.6);  
		\draw (10.2,8.58)--(10.2,0.6);
		
		

		\draw[thick, rotate=90] (9.6,0.03) ellipse (30pt and 30pt);
		\draw[thick, rotate=90] (9.6,-9.9) ellipse (30pt and 30pt);
		\draw[thick, rotate=90] (-0.4,0.03) ellipse (30pt and 30pt);
		\draw[thick, rotate=90] (-0.4,-9.9) ellipse (30pt and 30pt);
		
		\draw[thick, rotate=90] (4.5,0) ellipse (200pt and 80pt);
		\draw[thick, rotate=90] (4.5,-9.8) ellipse (200pt and 80pt);
		\end{tikzpicture}
		\caption{$FQ_n-M_2$ isomorphic to $Q_n$}
	\end{figure}
	 
	 Case $2: $ $M = M_1$.\\	
		Now consider $G_1 = FQ_n - M_1$. Define $f : Q_n \rightarrow  G_1$ as $f(x_1,x_2,....x_{n-1}, 0) = (x_1,x_2,....x_{n-1}, 0) $ and $f(x_1,x_2,....x_{n-1}, 1) = \overline{(x_1,x_2,....x_{n-1}}, 1)$.
	Suppose that $XY \in E(Q_n)$.\\
	$a: $ If $X,Y$ are vertices of $Q^0_{(n-1)}$ then $f(X)f(Y) \in E(Q^0_{(n-1)})  \subset E(G_1)$.\\
	$b:$ In case when $X,Y$ are vertices of $Q^1_{(n-1)}$ and $XY \in E(Q_n)$ means they differ in only one position say $i^{th}$ position $(1 \leq i \leq {n-1})$. But then $\overline{(x_1,x_2,....x_{n-1}}, 1)$ and $\overline{(y_1,y_2,....y_{n-1}}, 1)$ also differ in $i^{th}$ position and hence $f(X)f(Y) \in E(Q^1_{(n-1)})  \subset E(G_1)$.\\
	$c: $ Suppose $X$ is vertex of $Q^0_{(n-1)}$ and $Y$ is vertex of $Q^1_{(n-1)}$, then $XY \in E(Q_n)$ only if $x_i = y_i(1 \leq i \leq {n-1})$. Now $x_i = y_i(1 \leq i \leq {n-1})$ gives $f(x_1,x_2,....x_{n-1}, 0) = (x_1,x_2,....x_{n-1}, 0) = X $ and
	$f(y_1,y_2,....y_{n-1}, 1) = \overline{(x_1,x_2,....x_{n-1}}, 1) = \overline X $. Hence, $f(X)f(Y) \in M_2 \subset  E(G_1)$. (see FIGURE.$3$)\\
	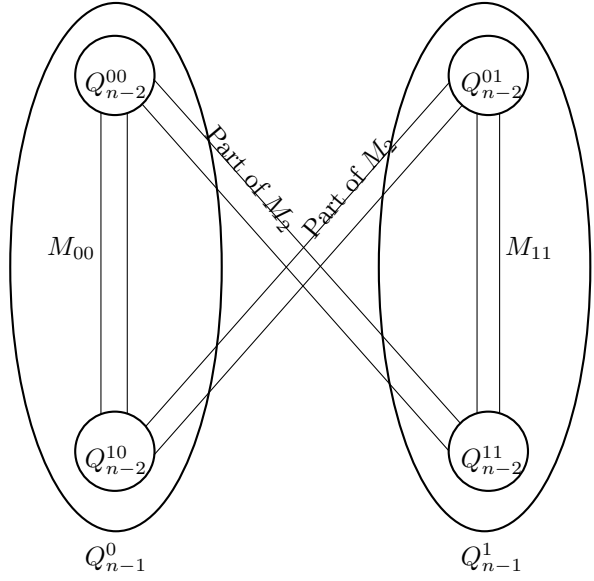
\begin{figure}[h]
		\begin{tikzpicture} [scale=.5]    
		
		\node [below] at (0,0) {$Q^{10}_{n-2}$}; 
		\node [below] at (0,10) {$Q^{00}_{n-2}$};  
		\node [below] at (10,10) {$Q^{01}_{n-2}$};  
		\node [below] at (10,0) {$Q^{11}_{n-2}$}; 
		
		\node [below] at (0,-2.6) {$Q^0_{n-1}$};  
		\node [below] at (10,-2.6) {$Q^1_{n-1}$};
		
		
		\node [below, rotate=50] at (5.8,7) {Part of $M_2$};
		\node [below, rotate=-50] at (4,7.2) {Part of $M_2$};
		
		\node [left] at (-0.3,5) {$M_{00}$};  
		\node [right] at (10.1,5) {$M_{11}$};
		
		
		
		\draw (-0.4,8.63)--(-0.4,0.57);  
		\draw (0.3,8.61)--(0.3,0.59);
		
		\draw (9.6,8.58)--(9.6,0.6);  
		\draw (10.2,8.58)--(10.2,0.6);
		
		\draw (0.78,0.25)--(8.85,9.4);
		\draw (1.01,-0.5)--(9.2,8.8);  
		
		\draw (9.17,0.33)--(1,9.5);
		\draw (8.86,-0.4)--(0.7,8.83);

		\draw[thick, rotate=90] (9.6,0.03) ellipse (30pt and 30pt);
		\draw[thick, rotate=90] (9.6,-9.9) ellipse (30pt and 30pt);
		\draw[thick, rotate=90] (-0.4,0.03) ellipse (30pt and 30pt);
		\draw[thick, rotate=90] (-0.4,-9.9) ellipse (30pt and 30pt);
		
		\draw[thick, rotate=90] (4.5,0) ellipse (200pt and 80pt);
		\draw[thick, rotate=90] (4.5,-9.8) ellipse (200pt and 80pt);
		\end{tikzpicture}
		\caption{$FQ_n-M_1$ isomorphic to $Q_n$}
	\end{figure}
	
	Case $3: $ $M = M_0$.\\		
		Let $G_2 = FQ_n - M_0$. Now define $f : Q_n \rightarrow  G_2$ as $f(x_1,x_2,....x_{n-2},1, j) = (\overline{x_1,x_2,....x_{n-2}},1,\overline j) $ and $f(x_1,x_2,....x_{n-2},0, j) = (x_1,x_2,....x_{n-2},0, j)$ ( $j \in \{0,1\}$ ).\\ For $j_1,j_2 \in \{0,1\}$.\\ 
		$a: $ If $X = (x_1,x_2,....x_{n-2},0, {j_1})$ and $Y = (y_1,y_2,....y_{n-2},0, {j_2}) $ then  $f(X)f(Y) \in E(Q^{00}_{(n-2)} \cup Q^{01}_{(n-2)} \cup M_1) \subset  E(G_1)$.\\ 
		$b: $ On other hand if $X = (x_1,x_2,....x_{n-2},1, {j_1})$  and $Y = (y_1,y_2,....y_{n-2},1, {j_2}) $  then $f(X)f(Y) \in E(Q^{10}_{(n-2)} \cup Q^{11}_{(n-2)} \cup M_1) \subset  E(G_1)$. \\ 
	 $c: $ Now, if $X = (x_1,x_2,....x_{n-2},0, {j_1})$ and $Y = (y_1,y_2,....y_{n-2},1, {j_2}) $ then $XY \in E(Q_n)$ only if $x_i = y_i(1 \leq i \leq {n-2})$ and ${j_1} = {j_2}$. But then $f(x) = (x_1,x_2,....x_{n-2},0, {j_1})$ and $f(Y) = (\overline{y_1,y_2,....y_{n-2}},1, \overline {j_2}) = (\overline{x_1,x_2,....x_{n-2}},1,\overline {j_1}) = \overline X $ gives $f(X)f(Y) \in M_2 \subset  E(G_1)$. (see FIGURE.$4$).\\
	 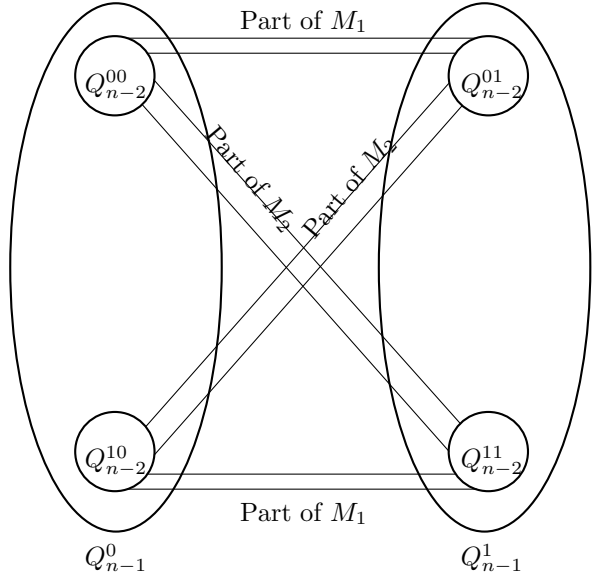
\begin{figure}[h]
	 	\begin{tikzpicture} [scale=.5]    
	 	
	 	\node [below] at (0,0) {$Q^{10}_{n-2}$}; 
	 	\node [below] at (0,10) {$Q^{00}_{n-2}$};  
	 	\node [below] at (10,10) {$Q^{01}_{n-2}$};  
	 	\node [below] at (10,0) {$Q^{11}_{n-2}$}; 
	 	
	 	\node [below] at (0,-2.6) {$Q^0_{n-1}$};  
	 	\node [below] at (10,-2.6) {$Q^1_{n-1}$};
	 	
	 	\node [above] at (5,10.5) {Part of $M_1$};
	 	\node [below] at (5,-1.5) {Part of $M_1$};
	 	
	 	\node [below, rotate=50] at (5.8,7) {Part of $M_2$};
	 	\node [below, rotate=-50] at (4,7.2) {Part of $M_2$};
	 	
	 	
	 	\draw (0.82,-1)--(9.03,-1);  
	 	\draw (0.3,-1.4)--(9.6,-1.4); 
	 	
	 	\draw (0.82,10.2)--(9.02,10.2);  
	 	\draw (0.3,10.6)--(9.6,10.6);
	 	
	 	
	 	
	 	\draw (0.78,0.25)--(8.85,9.4);
	 	\draw (1.01,-0.5)--(9.2,8.8);  
	 	
	 	\draw (9.17,0.33)--(1,9.5);
	 	\draw (8.86,-0.4)--(0.7,8.83);

	 	\draw[thick, rotate=90] (9.6,0.03) ellipse (30pt and 30pt);
	 	\draw[thick, rotate=90] (9.6,-9.9) ellipse (30pt and 30pt);
	 	\draw[thick, rotate=90] (-0.4,0.03) ellipse (30pt and 30pt);
	 	\draw[thick, rotate=90] (-0.4,-9.9) ellipse (30pt and 30pt);
	 	
	 	\draw[thick, rotate=90] (4.5,0) ellipse (200pt and 80pt);
	 	\draw[thick, rotate=90] (4.5,-9.8) ellipse (200pt and 80pt);
	 	\end{tikzpicture}
	 	\caption{$FQ_n-M_0$ isomorphic to $Q_n$}
	 \end{figure}

\end{proof}

Suppose after removing a subset $E^m$ of $2^n - 1$ edges of $FQ_n$ we get $FQ_n - E^m$ isomorphic to $Q_n$ then naturally $E^m$ is a perfect matching. Because $Q_n$ which is $n-$regular graph is a spanning subgraph of $FQ_n$ which is $(n+1)-$regular. Hence, $E^m$ is spanning subgraph of $FQ_n$ because after removing it reduces the degree of each vertex of $FQ_n$ by one. Means $E^m$ is spanning subgraph of disjoint edges with a degree of every vertex one in $E^m$, hence is perfect matching.\\
 Non-removability of perfect matching we prove by contradiction. While proving we will make use of the property of hypercube $Q_n$ that it is vertex-transitive and its diameter is $n$, due to these properties there exists exactly one vertex $u_1$ in $Q_n$ such that $d(u,u_1) = n$ for given any vertex $u$ in $Q_n$. Now we prove our main result.
 
\begin{thm} Let $n \geq 4$. In $FQ_n$, a perfect matching $M$ is non-removable if $M \neq M_i (0 \leq i \leq 2)$ and $M \subseteq \bigcup^2_{i=0} M_i$.  
\end{thm} 
 
\begin{proof} Let $M$ ($M \neq M_i (0 \leq i \leq 2)$ and $M \subseteq \bigcup^2_{i=0} M_i$) be a perfect matching in $FQ_n (n \geq 4)$ such that $FQ_n - M = G$(say) is isomorphic to $Q_n$. As $Q_n$ is vertex-transitive and its diameter is $n$ so $G$ is also vertex-transitive and its diameter is $n$.\\
	 For $0 \leq i \leq 2,$ $M \neq M_i$ gives $G\cap M_i \neq \phi $. Thus $G$ contains at least one hypercube edge and one augmented edge. Due to vertex symmetry without loss let vertex $u = (0,0, \ldots ,0,0 ) $ and vertex $v = (1,0, \ldots ,0,0 )$ (both from $ V(Q^{00}_{(n-2)}) $), are such that $u$ has its augmented edge $uu^c = e_1$ (vertex $ u^c = (1,1, \ldots ,1 ) \in V(Q^{11}_{(n-2)}) $) in $G$ and $v$  has its hypercube edge $vv^h = e_2$ (vertex $v^h = (1,0,0, \ldots ,0,0,0,1 ) \in V(Q^{01}_{(n-2)}) $) in $G$. \\
	Case $1:$ $M \subseteq \bigcup^2_{i=1} M_i$.\\
	 Now in $G$, we calculate the distance of all vertices from $ u$.\\
	  $a: $ As $M \subseteq \bigcup^2_{i=1} M_i$, therefore $ M_0 \in E(G)$.  Every vertex in $Q^0_{(n-1)}$ is at the distance at most $n-1$ from $u$.\\
$b: $ Due to the presence of an edge $e_1$, every vertex in $Q^{11}_{(n-2)}$ is at the distance at most $1 + (n-2) = n-1$ from $u$.\\ 
$c: $ Now we calculate distance of vertices (except vertex $u^h = (0,0, \ldots ,0,0,1 )$) of $ Q^{01}_{(n-2)} $ from the vertex $u$. Vertices of $ Q^{11}_{(n-2)} $ except vertex $w = (0,0, \ldots ,0,1,1 )$ are within a distance $n-3$ from $u^c$, and all these vertices are joined to their corresponding vertices of $ Q^{01}_{(n-2)} $ through $M_{11}$. Thus, $u$ can reach to  these vertices  by using the edge $e_1$ and sub hypercube edges $M_{11}$ within the distance $1 + (n-3) + 1 = n-1$.\\
The remaining vertex is $u^h$ which can be reached within distance $3 \leq {n-1}$ by using the path $u-v-v^h-u^h$.\\
Case $2:$ $M \subseteq \bigcup^2_{i=0} M_i$ and $M\cap M_0 \neq \phi $.\\
Because of all these conditions $M\cap M_{00} \neq \phi $ and $M\cap M_{11} \neq \phi $. Thus $M$ will contain at least one edge of $M_0$ say $e_3$ and $G$ will contain at least one edge of $M_0$ say $e_{4}$. Due to symmetry without loss let end vertices of the edge $e_{3}$ be $ u $ and $z = (0,0,0, \ldots 0,1,0 ) \in V(Q^{10}_{(n-2)})$. And end vertices of the edge $e_{4}$ be $v$ and $y = (1,0,0, \ldots 0,1,0 ) \in V(Q^{10}_{(n-2)})$. Let us denote by $e_5 \in E(G)$ the edge $uu^h$.\\  
$a: $ Due to presence of the edge $e_5$, every vertex in $Q^{01}_{(n-2)}$ is at the distance at most $1 + (n-2) = n-1$ from $u$.\\  
$b: $ Due to presence of the edge $e_1$, every vertex in $Q^{11}_{(n-2)}$ is at the distance at most $1 + (n-2) = n-1$ from $u$.\\
$c: $ Now we calculate distance of vertices (except vertex $z = (0,0, \ldots ,0,0,1,0 )$) of $ Q^{10}_{(n-2)} $ from the vertex $u$.The vertex $u$ can be reach to some of these vertices within distance $n-1$ by using sub hypercube edges of $M_{00} \cap G$ and remaining vertices except the vertex $z$ can be reached through the edge $e_1$ and then hypercube edges $M_{1} \cap E(G)$ whose end vertices (except vertex $w$) are in $V(Q^{11}_{(n-2)})$ within distance $1 + (n-3) + 1 = n-1$.\\ 
The remaining vertex is $z$ which can be reached within distance $3 \leq {n-1}$ by using the path $u-v-y-z$.\\
Thus, in both cases $d(u,u_1)= n-1$ for any vertex $u_1 \neq u \in V(G)$.(See FIGURE.5)\\
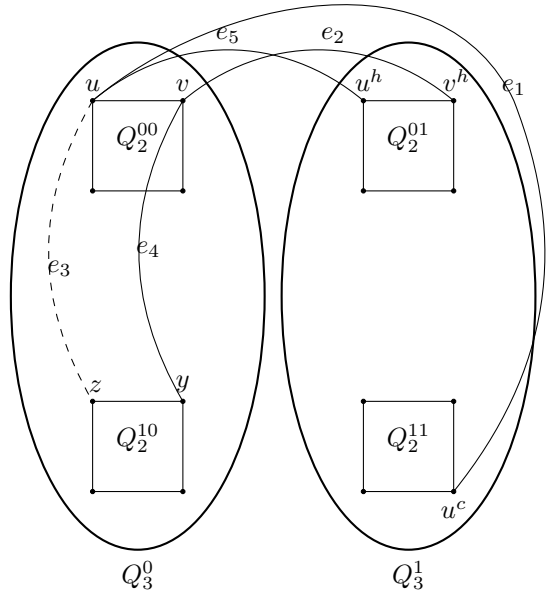
\begin{figure}[h]
	\begin{tikzpicture} [scale=.4]    
	
	\draw[fill=black](0,0) circle(.07); 
	\draw[fill=black](0,3) circle(.07); 
	\draw[fill=black](3,3) circle(.07); 
	\draw[fill=black](3,0) circle(.07); 
	
	\draw[fill=black](9,0) circle(.07); 
	\draw[fill=black](9,3) circle(.07); 
	\draw[fill=black](12,0) circle(.07); 
	\draw[fill=black](12,3) circle(.07); 
	
	\draw[fill=black](0,10) circle(.07); 
	\draw[fill=black](0,13) circle(.07); 
	\draw[fill=black](3,13) circle(.07); 
	\draw[fill=black](3,10) circle(.07); 
	
	\draw[fill=black](9,10) circle(.07); 
	\draw[fill=black](9,13) circle(.07); 
	\draw[fill=black](12,13) circle(.07); 
	\draw[fill=black](12,10) circle(.07); 
	
	
	
	
	
	\node [below] at (14,14) {$e_1$};
	\node [above] at (8,14.6) {$e_2$};
	\node [above] at (4.5,14.6) {$e_5$};
	\node [below] at (-1.1,8) {$e_3$};
	\node [above] at (1.85,7.5) {$e_4$};
	
	\node [above] at (0,13) {$u$};
	\node [above] at (3,13) {$v$};
	\node [above] at (9.2,13) {$u^h$};
	\node [above] at (12.1,13) {$v^h$};
	\node [above] at (0.1,3) {$z$};
	\node [above] at (3,3) {$y$};
	\node [below] at (12,0) {$u^c$};
	
	\node [above] at (1.5,11) {$Q^{00}_2$};
	\node [above] at (1.5,1) {$Q^{10}_2$};
	\node [above] at (10.5,11) {$Q^{01}_2$};
	\node [above] at (10.5,1) {$Q^{11}_2$};
	
	\node [below] at (1.5,-2) {$Q^0_3$};
	\node [below] at (10.5,-2) {$Q^1_3$};
	
	\draw (0,0)--(3,0)--(3,3)--(0,3)--(0,0);
	\draw (9,0)--(12,0)--(12,3)--(9,3)--(9,0);
	\draw (0,10)--(3,10)--(3,13)--(0,13)--(0,10);
	\draw (9,10)--(12,10)--(12,13)--(9,13)--(9,10);
	
	\draw (0,13) to[out=40, in=115] (14,13) to[bend left] (12,0);

	\draw[thick, rotate=90] (6.5,-1.5) ellipse (240pt and 120pt);
	\draw[thick, rotate=90] (6.5,-10.5) ellipse (240pt and 120pt);
	
	\draw (0,13) to[out=40, in=140] (9,13);
	\draw (3,13) to[out=40, in=140] (12,13);
	
	\draw[dashed] (0,13) to[out=-120, in=120] (0,3);
	\draw (3,13) to[out=-120, in=120] (3,3);
	
	\end{tikzpicture}
		\caption{$FQ_4 - M$ non-isomorphic to $Q_4$}
\end{figure}

Hence our assumption $G$ isomorphic to $Q_n$ is wrong.
  
 \end{proof}

 \noindent \textbf{Concluding remarks}\\
Thus we constructed perfect matchings of which some are removable and some are non-removable. The natural question arise here, whether we can classify all the perfect matchings of $FQ_n$ removal of which gives graph isomorphic or non-isomorphic to $Q_n$?~\\

 \noindent {\bf Acknowledgment:} The author gratefully
 acknowledges the Department of Science and Technology, New Delhi, India
 for the award of Women Scientist Scheme (SR/WOS-A/PM-79/2016) for research in Basic/Applied Sciences.

{\centerline{************}}

\bibliographystyle{amsplain}

\end{document}